\documentclass{amsart}%
\usepackage{amsmath}
\usepackage{amsfonts}
\usepackage{amssymb}
\usepackage{graphicx}%
\providecommand{\U}[1]{\protect\rule{.1in}{.1in}}
\newtheorem{theorem}{Theorem}
\theoremstyle{plain}

\newtheorem{corollary}{Corollary}

\newtheorem{definition}{Definition}

\newtheorem{lemma}{Lemma}

\newtheorem{proposition}{Proposition}

\numberwithin{equation}{section}
\begin{document}
\title[Rings satisfying $\ast$-property]{Rings satisfying $\ast$-property}
\author{K\"{u}r\c{s}at Hakan Oral}
\address{Y\i ld\i z Technical University,Department of mathematics Davutpa\c{s}a,34210,Esenler,Istanbul,TURKEY}
\email{khoral@yildiz.edu.tr}
\author{Bayram Ali Ersoy}
\curraddr{Y\i ld\i z Technical University,Department of mathematics Davutpa\c{s}a,34210,Esenler,Istanbul,TURKEY}
\email{ersoya@yildiz.edu.tr}
\author{\"{U}nsal Tekir}
\address{Marmara University, Department of mathematics, Ziverbey, 34722,Goztepe,Istanbul,TURKEY}
\email{utekir@marmara.edu.tr}
\thanks{This paper is in final form and no version of it will be submitted for
publication elsewhere.}
\date{June 16, 2013}

\begin{abstract}
In this paper we will investigate commutative rings which have the $\ast
$-property. We say that a ring $R$ satisfy $\ast-$property if for any family
of ideals $\left\{  I_{\alpha}\right\}  _{\alpha\in S}$ of $R$ in which $S$ is
an index set, there exists a finite subset\ $S^{\prime}$ of $S$ such that
$\sqrt{\underset{\alpha\in S}{%
{\displaystyle\bigcap}
}I_{\alpha}}=\underset{\alpha\in S^{\prime}}{%
{\displaystyle\bigcap}
}\sqrt{I_{\alpha}}$. We will show that any integral domain which satisfy
$\ast-$property is a field. Furthermore, these rings are zero-dimensional.
After this we give relations between these rings and Artinian rings.

\end{abstract}
\maketitle

\section{Introduction}

Throughout this paper, $R$ denotes a commutative ring with identity and the
set of all ideals of $R$ is denoted by $L\left(  R\right)  $.

M. Arapovic gives the following charecterization of imbeddibility of a ring
into a 0-dimensional ring: A ring $R$ is embeddable in a zero-dimensional ring
if and only if $R$ has a family of primary ideals $\left\{  Q_{\lambda
}\right\}  _{\lambda\in\Lambda}$ such that, $\left(  A1\right)  \underset
{\lambda\in\Lambda}{%
{\displaystyle\bigcap}
}Q_{\lambda}=0$ \ and $\left(  A2\right)  $ For each $a\in R$, there is an
$n\in%
\mathbb{N}
$ such that for all $\lambda\in\Lambda$, if $a\in\sqrt{Q_{\lambda}}$ then
$a^{n}\in Q_{\lambda}$ [2, Theorem 7]. Jim Brewer and Fred Richman [3] give an
equivalent condition for the condition (A2) in their paper titled "Subrings of
zero-dimensional rings" and characterize the condition (A2). A family
$\left\{  Q_{\lambda}\right\}  _{\lambda\in\Lambda}$ of ideals in a ring $R$
satisfies (A2) if and only if for each (countable) subset $\Gamma
\subseteq\Lambda$, $\sqrt{\underset{\lambda\in\Gamma}{%
{\displaystyle\bigcap}
}Q_{\lambda}}=\underset{\lambda\in\Gamma}{%
{\displaystyle\bigcap}
}\sqrt{Q_{\lambda}}$. With this motivation if we work this equivalent
condition, we see that some rings have further from this property. For an
example, if $R$ is an Artinian ring then for any subset $\left\{  I_{\alpha
}\right\}  _{\alpha\in S}$ of $L(R)$,%

\[
\sqrt{\underset{\alpha\in S}{%
{\displaystyle\bigcap}
}I_{\alpha}}=\underset{\alpha\in S^{\prime}}{%
{\displaystyle\bigcap}
}\sqrt{I_{\alpha}}%
\]
\ \ \ \ \ for some finite subset $S^{\prime}$ of $S$ [Corollary 1]. This
equality is called $\ast$-property. We call a ring $R$ satisfying $\ast
$-property if any family of ideals of the ring $R$ satisfies $\ast$-property.
After this we give the correspondence between such rings and Artinian rings.
Then if we use a result of the paper Subrings of zero-dimensional rings by Jim
Brewer and Fred Richman, we get that a ring satisfying $\ast$-property is a
zero-dimensional, in the sense of Krull dimension. Further we give an
equivalent condition with rings satisfying $\ast$-property, Artinian rings,
zero-dimensional rings and $\pi$-regular rings.

\section{Rings satisfying $\ast$-property}

\begin{definition}
Let $R$ be a ring and $S$ be an index set. If $\sqrt{\underset{\alpha\in S}{%
{\displaystyle\bigcap}
}I_{\alpha}}=\underset{\alpha\in S^{\prime}}{%
{\displaystyle\bigcap}
}\sqrt{I_{\alpha}}$ $\left(  \text{for some finite subset }S^{\prime}\text{of
}S\right)  $is satisfied for any arbitrary family $\left\{  I_{\alpha
}\right\}  _{\alpha\in S}$ of ideals of $R$, then $R$ is said to satisfy
$\ast$-property.\newline
\end{definition}

It is clear that every finite ring satisfies the $\ast$-property. For a
counterexample it is clear that the ring of integers does not satisfy $\ast$-property.

\begin{theorem}
If $R$ satisfies the d.c.c. for radical ideals then $R$ satisfies $\ast$-property.
\end{theorem}

\begin{proof}
Let $\left\{  I_{\alpha}\right\}  _{\alpha\in\Delta}$ be a family of ideals of
$R$. \ Now consider $\sqrt{\underset{\alpha\in\Delta}{\overset{}{%
{\displaystyle\bigcap}
}}I_{\alpha}}$ for any arbitrary family $\left\{  I_{\alpha}\right\}
_{\alpha\in S}$ of ideals of $R$ and $\Delta$ index set. Since $R $
\ satisfies the d.c.c for radical ideals, $\sqrt{\underset{\alpha\in\Delta
}{\overset{}{%
{\displaystyle\bigcap}
}}I_{\alpha}}=\sqrt{%
{\displaystyle\bigcap\limits_{i=1}^{n}}
I_{\alpha_{i}}}=%
{\displaystyle\bigcap\limits_{i=1}^{n}}
\sqrt{I_{\alpha_{i}}}$ \ for some finite subset $\left\{  \alpha_{1}%
,\alpha_{2},...,\alpha_{n}\right\}  $ of $\Delta$ where $n$ is a positive integer.
\end{proof}

\begin{corollary}
Every Artinian ring satisfies $\ast$-property.
\end{corollary}

\begin{theorem}
Let $R$ be an integral domain. If $R$ satisfies $\ast$-property, then $R$ is a field.
\end{theorem}

\begin{proof}
Let $I=\underset{\alpha\in S}{%
{\displaystyle\bigcap}
}I_{\alpha}$ where $I_{\alpha}$'s are all nonzero ideals of $R$ for $\alpha\in
S$. Then $I\neq(0)$. Indeed, if $I=(0),$ then%
\[
(0)=\sqrt{I}=\sqrt{\underset{\alpha\in S}{%
{\displaystyle\bigcap}
}I_{\alpha}}%
\]
and since $R$ is satisfying $\ast$-property we get%
\[
(0)=\sqrt{I}=\sqrt{\underset{\alpha\in S}{%
{\displaystyle\bigcap}
}I_{\alpha}}=\underset{\beta_{j}\in S^{\prime}}{%
{\displaystyle\bigcap}
}\sqrt{I_{\beta_{j}}}\text{ }%
\]
where $S^{\prime}=\left\{  \beta_{1},...,\beta_{n}\right\}  $ . Since all
$I_{\beta_{j}}\neq(0),$ there exists $0\neq a_{\beta_{j}}\in I_{\beta_{j}} $
for all $\beta_{j}\in S^{\prime}$. Therefore,%
\[
0\neq a=a_{\beta_{1}}a_{\beta_{2}}...a_{\beta_{n}}\in\underset{\beta_{j}\in
S^{\prime}}{%
{\displaystyle\bigcap}
}I_{\beta_{j}}\subset\underset{\beta_{j}\in S^{\prime}}{%
{\displaystyle\bigcap}
}\sqrt{I_{\beta_{j}}}=\sqrt{I}%
\]
which is a contradiction. Thus $I\neq(0)$. Choose a nonzero element $x\in I$.
Since $I$ is the unique smallest non-zero ideal, we get that $(x)=I=(x^{2}) $.
Thus $x=rx^{2}$ for some $r\in R$. This shows that $x$ is a unit. Hence $R $
is a field.
\end{proof}

\begin{lemma}
Let $R$ be a ring satisfying $\ast$-property. Then every homomorphic image of
$R$ satisfies $\ast$-property.
\end{lemma}

\begin{proof}
Let $f:R\rightarrow R^{\prime}$ be a surjective homomorphism, $R$ is
satisfying $\ast-property$ and $\left\{  I_{\alpha}^{\prime}\right\}
_{\alpha\in S}$ is an arbitrary family of ideals of $R^{\prime}$. Then there
exists a family of ideals $\left\{  I_{\alpha}\right\}  _{\alpha\in S}$ such
that $f\left(  I_{\alpha}\right)  =I_{\alpha}^{\prime}$ and $\ker(f)\subseteq
I_{\alpha}$. Then we get,%
\[
\sqrt{\underset{\alpha\in S}{%
{\displaystyle\bigcap}
}I_{\alpha}^{\prime}}=\sqrt{\underset{\alpha\in S}{%
{\displaystyle\bigcap}
f}\left(  I_{\alpha}\right)  }=\sqrt{f\left(  \underset{\alpha\in S}{%
{\displaystyle\bigcap}
}I_{\alpha}\right)  }=f\left(  \sqrt{\underset{\alpha\in S}{%
{\displaystyle\bigcap}
}I_{\alpha}}\right)  .
\]
Since $R$ satisfies $\ast$-property, there exists a finite subset $S^{\prime}$
of $S$ \ such that $\sqrt{\underset{\alpha\in S}{%
{\displaystyle\bigcap}
}I_{\alpha}}=\underset{\alpha\in S^{\prime}}{%
{\displaystyle\bigcap}
}\sqrt{I_{\alpha}}$. And so%
\[
\sqrt{\underset{\alpha\in S}{%
{\displaystyle\bigcap}
}I_{\alpha}^{\prime}}=f\left(  \underset{\alpha\in S^{\prime}}{%
{\displaystyle\bigcap}
}\sqrt{I_{\alpha}}\right)  =\underset{\alpha\in S^{\prime}}{%
{\displaystyle\bigcap}
f}\left(  \sqrt{I_{\alpha}}\right)  =\underset{\alpha\in S^{\prime}}{%
{\displaystyle\bigcap}
}\left(  \sqrt{f\left(  I_{\alpha}\right)  }\right)  =\underset{\alpha\in
S^{\prime}}{%
{\displaystyle\bigcap}
}\sqrt{I_{\alpha}^{\prime}}\text{.}%
\]

\end{proof}

\begin{corollary}
Let $R$ be a ring satisfying $\ast$-property and $I$ be an ideal of $R$. Then
$R/I$ satisfies $\ast$-property.
\end{corollary}

\begin{proposition}
Let $R$ and $R^{\prime}$be rings satisfying $\ast$-property. Then $R\times
R^{\prime}$ satisfies $\ast$-property.
\end{proposition}

\begin{proof}
It is clear.
\end{proof}

M. Arapovic, in his paper "On the embedding of a commutative ring into a
0-dimensional Commutative ring" gives the following charecterization of
imbeddibility of a ring into a 0-dimensional ring [2, Theorem 7].

\begin{theorem}
\lbrack2, Theorem 7] A ring $R$ is embeddable in a zero-dimensional ring if
and only if $R$ has a family of primary ideals $\left\{  Q_{\lambda}\right\}
_{\lambda\in\Lambda}$ such that\newline$A1$. $\underset{\lambda\in\Lambda}{%
{\displaystyle\bigcap}
}Q_{\lambda}=0$ and\newline$A2$. For each $a\in R$, there is $n\in%
\mathbb{N}
$ such that for all $\lambda\in\Lambda$, if $a\in\sqrt{Q_{\lambda}}$, then
$a^{n}\in Q_{\lambda}$.
\end{theorem}

In the paper "Subrings of zero-dimensional rings", Jim Brewer and Fred Richman
give an equivalent condition for (\textit{A2}) and characterize the condition
[3, Theorem 2].

\begin{theorem}
\lbrack3, Theorem 2] A family $\left\{  Q_{\lambda}\right\}  _{\lambda
\in\Lambda}$ of ideals in a ring $R$ satisfies (A2) if and only if for each
(countable) subset $\Gamma\subseteq\Lambda$\newline%
\[
\sqrt{\underset{\lambda\in\Gamma}{%
{\displaystyle\bigcap}
}Q_{\lambda}}=\underset{\lambda\in\Gamma}{%
{\displaystyle\bigcap}
}\sqrt{Q_{\lambda}}\text{ .}%
\]

\end{theorem}

\begin{theorem}
\lbrack3, Theorem 4] The following conditions on a ring $R$ are
equivalent:\newline(i) The ring $R$ is zero dimensional,\newline(ii) Condition
(A2) holds for the family of all ideals of $R$,\newline(iii) Condition (A2)
holds for the family of all primary ideals of $R$.
\end{theorem}

With this Theorem we give the following Proposition, which gives us that a
ring satisfying $\ast$-property satisfies the condition (A2). And so we get
that every ring satisfying $\ast$-property is a zero-dimensional ring.

\begin{proposition}
Let $R$ be a ring satisfying $\ast$-property and $\left\{  Q_{\alpha}\right\}
_{\alpha\in S}$ be a family of ideals. Then $\left\{  Q_{\alpha}\right\}
_{\alpha\in S}$ satisfies the (A2) condition.
\end{proposition}

\begin{proof}
By assumption for any countable subset $T\subseteq S$,
\[
\sqrt{\underset{\alpha\in T}{%
{\displaystyle\bigcap}
}Q_{\alpha}}=\underset{\alpha\in T^{\prime}}{%
{\displaystyle\bigcap}
}\sqrt{Q_{\alpha}}%
\]
for some finite subset $T^{\prime}$ of $T$. Then
\[
\underset{\alpha\in T^{\prime}}{%
{\displaystyle\bigcap}
}\sqrt{Q_{\alpha}}=\sqrt{\underset{\alpha\in T}{%
{\displaystyle\bigcap}
}Q_{\alpha}}\subseteq\underset{\alpha\in T}{%
{\displaystyle\bigcap}
}\sqrt{Q_{\alpha}}\subseteq\underset{\alpha\in T^{\prime}}{%
{\displaystyle\bigcap}
}\sqrt{Q_{\alpha}}%
\]
and so we get
\[
\sqrt{\underset{\alpha\in T}{%
{\displaystyle\bigcap}
}Q_{\alpha}}=\underset{\alpha\in T}{%
{\displaystyle\bigcap}
}\sqrt{Q_{\alpha}}\text{.}%
\]
Thus by Theorem 4 we get the result.
\end{proof}

\begin{corollary}
Every ring which satisfies $\ast$-property is zero-dimensional.
\end{corollary}

\begin{proof}
It follows from Proposition 2 and Theorem 5.
\end{proof}

Here we recall that a commutative ring with identity is called von Neumann
regular ring (VNR) if for all $x\in R$ there exists $y\in R$ such that
$x^{2}y=x$. Also we give the following characterization of VNR which was given
by R.Gilmer in [4].

\begin{theorem}
\lbrack4, Theorem 3.1] The following conditions are equivalent in a ring
$R$:\newline(i) $R$ is VNR,\newline(ii) $R$ is zero-dimensional and
reduced,\newline(iii) $R_{P}$ is a field for each $P\in Spec(R)$,\newline(iv)
Each ideal of $R$ is a radical ideal,\newline(v) Each ideal of $R$ is idempotent.
\end{theorem}

\begin{corollary}
Let $R$ be a reduced ring. If $R$ is satisfying $\ast$-property then $R$ is VNR.
\end{corollary}

\begin{proof}
It follows from Corollory 3 and Theorem 6.
\end{proof}

\begin{lemma}
Let $R$ be a ring satisfying $\ast$-property. Consider the following
conditions:\newline(i) $R$ is a VNR,\newline(ii) If $\sqrt{I}=\sqrt{J}$ for
some ideals $I$, $J$ of $R$ then $I=J$.\newline If one of the above conditions
holds, then $R$ is an Artinian ring.
\end{lemma}

\begin{proof}
Let $R$ be a ring satisfying $\ast$-property and%
\[
I_{1}\supseteq I_{2}\supseteq I_{3}\supseteq...
\]
be a descending chain of ideals in $R$. Then we get $\sqrt{\underset{\alpha\in
S}{%
{\displaystyle\bigcap}
}I_{\alpha}}=\underset{i=1}{\overset{n}{%
{\displaystyle\bigcap}
}}\sqrt{I_{\alpha_{i}}}$ for some finite subset $S^{\prime}=\left\{
\alpha_{1},...,\alpha_{n}\right\}  $ of $S$, where $I_{\alpha_{1}}%
\supseteq...\supseteq I_{\alpha_{n}}$. And so%
\[
\sqrt{\underset{\alpha\in S}{%
{\displaystyle\bigcap}
}I_{\alpha}}=\underset{i=1}{\overset{n}{%
{\displaystyle\bigcap}
}}\sqrt{I_{\alpha_{i}}}=\sqrt{\underset{i=1}{\overset{n}{%
{\displaystyle\bigcap}
}}I_{\alpha_{i}}}=\sqrt{I_{\alpha_{n}}}\text{.}%
\]
Thus by our hypothesis we get $\underset{\alpha\in S}{%
{\displaystyle\bigcap}
}I_{\alpha}=I_{\alpha_{n}}$. The chain must stop.
\end{proof}

Recall that a commutative ring with identity $R$ is called Laskerian if each
ideal of $R$ is a finite intersection of primary ideals.

\begin{theorem}
Let $R$ be a commutative ring with identity. If $R$ is a Laskerian ring then
the following conditions are equivalent:\newline(i) $R$ is a ring satisfying
$\ast$-property,\newline(ii) For the family $\left\{  P_{\lambda}\right\}
_{\lambda\in\Lambda}$ of prime ideals of $R$ we get$\sqrt{\underset{\lambda
\in\Lambda}{%
{\displaystyle\bigcap}
}P_{\lambda}}=\underset{\lambda\in\Gamma}{%
{\displaystyle\bigcap}
}P_{\lambda}$ for a finite subset $\Gamma$ of $\Lambda$ .
\end{theorem}

\begin{proof}
(i)$\Rightarrow$(ii) It is clear.\newline(ii)$\Rightarrow$(i) Suppose that
$\sqrt{\underset{\lambda\in\Lambda}{%
{\displaystyle\bigcap}
}P_{\lambda}}=\underset{\lambda\in\Gamma}{%
{\displaystyle\bigcap}
}P_{\lambda} $ holds for a finite subset $\Gamma$ of $\Lambda$. Let $\left\{
I_{\alpha}\right\}  _{\alpha\in\Delta}$ be a family of ideals of $R$. Since
$R$ is Laskerian, we have $\sqrt{I_{\alpha}}=\underset{i=1}{\overset{n}{%
{\displaystyle\bigcap}
}}P_{\alpha,i}$. Thus we get $\sqrt{\underset{\alpha,i}{%
{\displaystyle\bigcap}
}P_{\alpha,i}}=\underset{finite}{%
{\displaystyle\bigcap}
}\sqrt{P_{\alpha,i}}=\underset{finite}{%
{\displaystyle\bigcap}
}P_{\alpha,i}$ and so, $\underset{finite}{%
{\displaystyle\bigcap}
}\sqrt{I_{\alpha}}=\underset{finite}{%
{\displaystyle\bigcap}
}P_{\alpha,i}=\underset{finite}{%
{\displaystyle\bigcap}
}\sqrt{P_{\alpha,i}}=\sqrt{\underset{\alpha,i}{%
{\displaystyle\bigcap}
}P_{\alpha}}\subseteq\sqrt{\underset{\alpha}{%
{\displaystyle\bigcap}
}I_{\alpha}}$. Since the converse inclusion holds always, we get
$\sqrt{\underset{\alpha}{%
{\displaystyle\bigcap}
}I_{\alpha}}=\underset{finite}{%
{\displaystyle\bigcap}
}\sqrt{I_{\alpha}}$. Hence $R$ satisfies $\ast$-property.
\end{proof}

\begin{corollary}
Let $R$ be a Laskerian ring which satisfies $\ast$-property and $S$ be a
multiplicatively closed subset of $R$. Then $S^{-1}R$ satisfies $\ast$-property.
\end{corollary}

\begin{proof}
Let $\left\{  \mathfrak{P}_{\alpha}\right\}  $ be a family of prime ideals of
$S^{-1}R$. Then there exists a family of prime ideals $\left\{  P_{\alpha
}\right\}  $ of $R$\ such that $\left(  P_{\alpha}\right)  ^{e}=\mathfrak{P}%
_{\alpha}$. And we have $%
{\displaystyle\bigcap}
\left(  P_{\alpha}\right)  ^{e}=\left(
{\displaystyle\bigcap}
P_{\alpha}\right)  ^{e}$ since $P_{\alpha}$ are prime ideals of $R$. Thus we
get%
\begin{align*}
\sqrt{\underset{\lambda\in\Lambda}{%
{\displaystyle\bigcap}
}\mathfrak{P}_{\lambda}}  &  =\sqrt{\underset{\lambda\in\Lambda}{%
{\displaystyle\bigcap}
}\left(  P_{\lambda}\right)  ^{e}}=\sqrt{\left(  \underset{\lambda\in\Lambda}{%
{\displaystyle\bigcap}
}P_{\lambda}\right)  ^{e}}=\left(  \sqrt{\underset{\lambda\in\Lambda}{%
{\displaystyle\bigcap}
}P_{\lambda}}\right)  ^{e}\\
&  =\left(  \underset{\lambda\in\Gamma}{%
{\displaystyle\bigcap}
}\sqrt{P_{\lambda}}\right)  ^{e}=\left(  \underset{\lambda\in\Gamma}{%
{\displaystyle\bigcap}
}P_{\lambda}\right)  ^{e}=\underset{\lambda\in\Gamma}{%
{\displaystyle\bigcap}
}\left(  P_{\lambda}\right)  ^{e}\\
&  =\underset{\lambda\in\Gamma}{%
{\displaystyle\bigcap}
}\mathfrak{P}_{\lambda}=\underset{\lambda\in\Gamma}{%
{\displaystyle\bigcap}
}\sqrt{\mathfrak{P}_{\lambda}}%
\end{align*}
\newline for a finite subset $\Gamma$ of $\Lambda$ . Thus $S^{-1}R$ satisfies
$\ast$-property.
\end{proof}

Now, if we let a Noetherian ring instead of Laskerian ring then we can give
the following theorem from the previous results. Before the main theorem we
will recall that a ring $R$ is called $\pi$-regular provided for all $a\in R $
there exist $a^{\prime}\in R$, $n\in%
\mathbb{N}
^{\ast}$ such that $a^{n}=\left(  a^{n}\right)  ^{2}a^{\prime}$ . In his paper
"Characterization of the 0-dimensional rings", Arapovic has show that a ring
is $\pi$-regular if and only if the ring is zero dimensional [1, Theorem 6].

\begin{theorem}
Let $R$ be a Noetherian ring. The following conditions are equivalent:\newline%
(i) $R$ satisfies $\ast$-property,\newline(ii) $R$ is
zero-dimensional,\newline(iii) $R$ is Artinian,\newline(iv) $R$ is $\pi
$-regular,\newline(v) For the family $\left\{  P_{\lambda}\right\}
_{\lambda\in\Lambda}$ of prime ideals of $R$ we get $\sqrt{\underset
{\lambda\in\Lambda}{%
{\displaystyle\bigcap}
}P_{\lambda}}=\underset{\lambda\in\Gamma}{%
{\displaystyle\bigcap}
}P_{\lambda}$ for a finite subset $\Gamma$ of $\Lambda$ .
\end{theorem}

\begin{proof}
$(i)\Rightarrow(ii)$ follows from Corollary 3,\newline$(ii)\Rightarrow(iii)$
follows from [5, Proposition 8.38],\newline$(iii)\Rightarrow(i)$ follows from
Corollary 1,\newline$(ii)\Leftrightarrow(iv)$ follows from [1, Theorem
6],\newline$(i)\Leftrightarrow(v)$ follows from Theorem 7.
\end{proof}


\begin{thebibliography}{9}                                                                                                %


\bibitem {A83a}Arapovic, M., Characterization of the 0-dimensional rings,
Glasnik Mathematicki, Vol. 18 (38), 1983, 39-46.

\bibitem {A83b}Arapovic, M., On the embedding of a commutative ring into a
0-dimensional commutative ring, Glasnik Mathematicki, Vol. 18 (38), 1983, 53-59.

\bibitem {BreR06}Brewer, J. and Richman, F., Subrings of zero-dimensional
rings, Multiplicative Ideal Theory in Commutative Rings, Springer, 2006, 73-88.

\bibitem {G95}Gilmer, R., Backhround and preliminaires on zero-dimensional
rings, Zero Dimensional Rings, Lecture Notes in Pure and Applied Mathematics,
171, Marcel Dekker 1995, 1-13.

\bibitem {S00}Sharp, R. Y., Steps in Commutative Algebra, Cambridge University
Press, Second edition 2000.
\end{thebibliography}
\end{document}